\documentclass[a4paper,12pt,reqno]{amsart}
\usepackage{amsmath,amsfonts,amssymb,amsthm,enumerate,multicol}
\usepackage{tikz,graphicx}
\usepackage{hyperref}
\usepackage{caption,enumitem,wasysym}
\usepackage{cleveref}
\usepackage{epstopdf}
\usepackage{float,wrapfig}
\usepackage[labelsep=space]{caption}
\usepackage[font=footnotesize]{caption}
\usepackage{xcolor}

\newtheorem{definition}{Definition}[section]
\newtheorem{theorem}{Theorem}[section]
\newtheorem{corollary}{Corollary}[section]

\newtheorem{lemma}{Lemma}[section]

\newtheorem{remark}{Remark}[section]

\allowdisplaybreaks
 \theoremstyle{plain}
\newtheorem{thm}{Theorem}[section]

\newtheorem{prop}[thm]{Proposition}

\DeclareMathOperator{\sgn}{sgn}

\usepackage{mathtools}

\usepackage{amsmath,amsfonts,amssymb,amsthm,enumerate,multicol}
\usepackage{tikz}
\usepackage{float}
\allowdisplaybreaks

\numberwithin{equation}{section}

\begin{document}

\begin{center}
\Large{\textbf{On the discrete Safronov–Dubovski\v{i} coagulation equation: well-posedness, mass-conservation and asymptotic behaviour}}
\end{center}
\medskip
\medskip
\centerline{${\text{Mashkoor~Ali}}$, ${\text{Pooja~Rai}}$ and   ${\text{Ankik ~ Kumar ~Giri$^*$}}$ }\let\thefootnote\relax\footnotetext{$^*$Corresponding author. Tel +91-1332-284818 (O);  Fax: +91-1332-273560  \newline{\it{${}$ \hspace{.3cm} Email address: }}ankik.giri@ma.iitr.ac.in}
\medskip
{\footnotesize

  \centerline{ ${}^{}$Department of Mathematics, Indian Institute of Technology Roorkee,}
   \centerline{Roorkee-247667, Uttarakhand, India}

}

\bigskip

\
\begin{quote}
{\small {\em \bf Abstract.} The global existence of mass-conserving weak solutions to the Safronov--Dubovski\v{i} coagulation equation is shown for the coagulation kernels satisfying the at most linear growth for large sizes. In contrast to previous works, the proof mainly relies on the de la Vall\'{e}e--Poussin theorem \cite[Theorem 7.1.6]{BLL 2019}, which only requires the finiteness of the first moment of the initial condition. By showing the necessary regularity of solutions, it is shown that the weak solutions constructed herein are indeed classical solutions. Under additional restrictions on the initial data, the uniqueness of solutions is also shown. Finally, the continuous dependence on the initial data and the large-time behaviour of solutions are also addressed.}
\end{quote}

\vspace{.3cm}

\noindent
{\rm \bf Mathematics Subject Classification(2020).} Primary: 34A12, 34K30, 47J35; Secondary: 46B50.\\

{ \bf Keywords:} Safronov–Dubovskii coagulation equation, Existence, Uniqueness, Mass conservation, large-time behaviour.\\



\section{\textbf{INTRODUCTION}}

The discrete coagulation equation is an infinite set of ordinary differential equations for the dynamics of cluster growth that describes the mechanism allowing clusters to undergo coagulation as the only event. We restrict ourselves to binary coagulation, which means two clusters combine to form a bigger one. It is assumed that the clusters are fully identified by their size/volume (a positive integer). A cluster of size $i$ (or $i$-cluster) is made of $i$ identical elementary particles known as monomers. The discrete coagulation equation describes the evolution of the concentration $\xi_i(t)$, $ i \in \mathbb{N}/\{0\}$ of clusters of size $i$ (or $i$-mers) at time $t \geq  0$ and can be written as the nonlinear nonlocal equation of the form
\begin{align}
\frac{d\xi_i}{dt}&= \frac{1}{2}\sum_{j=1}^{i-1} \gamma_{j,i-j} \xi_j \xi_{i-j} -\sum_{j=1}^{\infty} \gamma_{i,j} \xi_i \xi_j \label{SCE}\\
\xi_i(0) &= \xi_i^{in}\geq 0 \label{SCEIC}
\end{align}
for $i \geq 1$.  Here $\gamma_{i,j}$ is the coagulation rate at which clusters of size $i$ merges with the clusters of size $j$ to form the larger clusters. It is assumed that $\gamma_{i,j}$ is nonnegative and symmetric, that is, $ 0\leq \gamma_{i,j} = \gamma_{j,i} \ \ \forall \ \ i, j \ge 1$.
The first term on the right-hand side of \eqref {SCE} accounts for the formation of $i$-clusters by binary coalescence of smaller ones, while the second term accounts for their depletion through coagulation with other clusters. In \cite{SMOL 1916, SMOL 1917}, Smoluchowski initially introduced a system of mathematical equations of the form \eqref{SCE}--\eqref{SCEIC} which is later referred to as \emph{Smoluchowski coagulation equation}  describing the coagulation of colloids moving in a Brownian motion. The system of equations given by \eqref{SCE}--\eqref{SCEIC} has been extensively studied in the presence of the fragmentation term, using two different techniques. If the primary focus is on the effects of strong coagulation, such as gelation, it becomes necessary to impose assumptions that ensure the coagulation term dominates over other processes. In such cases, the analysis requires the use of weak compactness arguments and working with weak solutions, as demonstrated in works such as \cite{BALL 90, CARR 92, COSTA 94, COSTA 95, PL 1999, PL 2002}. Alternatively, if the process is driven by the linear fragmentation part, it allows for more flexibility in selecting the fragmentation and transport parts, resulting in strong, classical solutions obtained within the framework of the semigroup theory. This has been shown in works such as \cite{BAN 12, BANLAM 12, BAN 18}. This discrete system and its continuous counterpart   have received considerable attention  in the mathematics and physics literature in recent years; due to the enormous number of works devoted to them, we refer to the classical review \cite{RLD 1972} and the more recent \cite{BLL 2019} for an overview of the topic.

 \par
In this article, we are mainly concerned with coagulation models with disperse systems, and the application of such models can be found in astrophysics (formation of stars and planets), chemistry (reacting polymers), meteorology (formation of clouds), and physics (growth of gas bubbles in solids). In \cite{PBD 1999}, Dubovski\v{i} looked at a dispersed system and introduced a model known as the Safronov-Dubovski\v{i} coagulation model, in which only binary collisions between particles can happen simultaneously, and the mass of each particle is also assumed to be proportional to some $m_0 > 0$. Collisions between particles of mass $im_0$ and $jm_0$ cause particles to grow in the system. Particles with mass $im_0$ will be referred to as $i$-mers, with $m_0$ being the mass of the smallest particle in the system. In this model, a collision between an $i$-mer and a $j$-mer causes the $j$-mer to split into $j$ monomers if $j\leq i$. Hence, we have another characterization of the coagulation process that leads to the balanced equation, also known as the discrete Safronov-Dubovski coagulation equation(DSDCE) is of the form 
\begin{align}
\frac{d\xi_i(t)}{dt}&= \xi_{i-1}(t) \sum_{j=1}^{i-1} j \gamma_{i-1,j} \xi_j(t)-\xi_i(t)  \sum_{j=1}^{i} j \gamma_{i,j} \xi_j(t)-\sum_{j=i}^{\infty} \gamma_{i,j} \xi_i(t) \xi_j(t),\label{DSDE}\\
\xi_i(0) &= \xi_i^{in} \geq 0, \hspace{.2cm}  \label{IC}
\end{align}
for $i \geq 1$.
Note that equation \eqref{DSDE}-\eqref{IC} is a nonlinear initial value problem that describes the dynamics of evolution of the concentration $\xi_i(t)$, $ i \in \mathbb{N}/\{0\}$ of clusters of size $i$ at time $t\ge 0$. The coagulation kernel, $\gamma_{i,j}$ (with $i \neq j$), specifies the rate at which $i$-mers collide with $j$-mers. The first sum in \eqref{DSDE} describes the $i$-mer introduction into the system as a result of collisions between $(i-1)$-mers and monomers produced from fragmented $j$-mers. If $i = 1$, the initial sum is zero. The second sum represents the loss (or decay) of $i$-mers due to monomer merging. The first and second sums are multiplied by $j$ to demonstrate that the collision involves exactly $j$ monomers. The third sum represents the decay of $i$-mers due to fragmentation caused by collisions with bigger particles.
\par
In \cite{PBD 1999}, Dubovskii obtained  the coagulation model proposed by Safronov \cite{SAFR 1972} from DSDCE \eqref{DSDE}--\eqref{IC}, which is given as
\begin{align}
\frac{\partial f}{\partial t} = -\frac{\partial}{\partial x} \Big[f(x,t) \int_0^x y a(x,y) f(y,t) dy \Big]- f(x,t) \int_x^{\infty}  a(x,y) f(y,t) dy \label{COHS}
\end{align}
The distribution  function $f(x,t)$ denotes the distribution of particle of size $x \in (0, \infty)$ at time $t\ge 0$ and the coagulation kernel $a(x,y)$, satisfying $a(x,y)=a(y,x)\ge0, \forall x, y \in (0,\infty)^2$, determines the rate at which particles of $x$ and $y$ coalesce. Using the classical weak-$L^1$ compactness technique, the existence of weak solutions to \eqref{COHS}, with suitable initial data, has been shown in \cite{BPA 2022,LLW 2003} whereas the self similar solutions have been discussed in \cite{BP 2007, PL 2005, PL 2006}.\par
Coming back to equations \eqref {DSDE}--\eqref {IC}, several results are available dealing with the solutions' existence, uniqueness, and mass conservation property. Existence and uniqueness of weak solutions to  \eqref {DSDE}--\eqref {IC} have been discussed in \cite{BAG 2005}, when the coagulation kernels satisfy $\lim_{j \to \infty} \frac{\gamma_{i, j}}{j }=0, \hspace{.2cm} i, j \geq 1$ and $\xi^{in} \in l_1$. Also,  a connection between \eqref {DSDE} and \eqref {COHS} has been established through a suitable sequence of solutions to \eqref {DSDE}. In \cite{DAV 2014}, the global existence of the classical solutions was demonstrated when the coagulation kernel satisfy $j\gamma_{i, j} \leq M, \hspace{.1cm} j \leq i$ and for the kernel of the form $\gamma_{i, j} \leq C_{\gamma} h_i h_j$ with an additional assumption $\frac{h_i}{i }\to 0$. Furthermore, mass conservation property for $\gamma_{i, j} \leq C_{\gamma} h_i h_j$ for $h_i \leq i^{\frac{1}{2}} $ and uniqueness for bounded kernels, i.e., $\gamma_{i, j}\leq C_{\gamma}, \hspace{.1cm} i, j\geq 1$ have been discussed. Recently, in \cite{DAS 2021}, global existence and mass-conservation property of solutions have been established for $ \gamma_{i,j} \leq (1+i +j)^{\alpha}$ where $ \alpha \in [0,1]$, whereas uniqueness of solutions is shown under the condition that $\gamma_{i,j} \leq Ci^{\kappa}$ for $ i \leq j$, where $ \kappa \leq 2$, in the weighted $l_1$ space i.e. a sequence $\xi= (\xi_i)_{i \geq 1} \in l_1$ having $ \sum_{i=1}^{\infty} i |\xi_i| < \infty$. In \cite{SR 2022}, the authors establishes the existence, uniqueness, and mass conservation of \eqref{DSDE}--\eqref{IC} when dealing with an unbounded kernel in the form of $\min\{i,j\}\gamma_{i,j}\le(i+j)$ for all $i,j \ge 1$ in the weighted $l_1$ space.  In the same space, recently in \cite{ALI 23}, the existence of \eqref{DSDE}--\eqref{IC} is proved for the coagulation coefficients of a multiplicative type, which are defined as follows:
\begin{align*}
\gamma_{i,j} = \theta_i \theta_j +\kappa_{i,j}.
\end{align*}
Moreover, these coefficients satisfy the following conditions:
\begin{align*}
\inf_{i\ge 1} \frac{\theta_i}{i} =B>0, ~~~~~~~~~ \text{and}~~~~~~~~~~\kappa_{i,j} \le A\theta_i \theta_j~~~~~\text{for each} ~~~~~i,j\ge 1 (A\ge 0).
\end{align*}
\par

Next, we define the moments of the concentration $\xi= (\xi_i(t))_{i\geq 1}$ of order $m\geq 0$ as
\begin{align}
  M_m(\xi(t)):= M_m(t) = \sum_{i=1}^{\infty} i^m \xi_i(t),
\end{align}
where the zeroth ($m=0$) and the first ($m=1$) moments denote the total number of particles and total mass of particles in the system, respectively. Observe that, since particles are neither created nor destroyed in the reactions described by \eqref{SCE} and \eqref{DSDE}, the total mass is expected to be conserved through the time evolution. Because the DSDCE \eqref{DSDE}--\eqref{IC}  only accounts for coagulation processes, the total number of particles (which is nothing but the $l_1$-norm of $\xi=(\xi_i)_{i\geq 1}$) is supposed to decrease to zero as time increases to infinity which is shown in Section \ref{SEC6}.
The current article improves on the results obtained in \cite{DAS 2021}, as several flaws were identified and stated below. The proof of the local existence theorem demonstrated that the truncated mass-conservation law holds, which is incorrect. Furthermore, the finiteness of the second moment of initial data has been employed to prove the existence and uniqueness of the solution, while the first moment of initial data is assumed to be finite. In fact, the uniqueness of the solution is examined for the coagulation kernels that do not overlap with the one for which the existence of a solution is shown, as far as unbounded kernels are concerned. Last but not least, the authors, in \cite{DAS 2021}, claimed to have established the existence of a global classical solution to \eqref{DSDE}--\eqref{IC} in the sense of the definition given for mild solution. Hence, the novelty of our work is that we have improved these results in several ways. We have refer the local existence result from \cite[Lemma 13]{BAG 2005}, and proved the existence of global classical solution to \eqref{DSDE}--\eqref{IC} in a weighted $l_1$  space without assuming the finiteness of the second moment. Furthermore, the uniqueness investigated for the same class of coagulation kernels as the one used to prove its existence.

\par 
The content of the paper is organized as follows. In Section \ref{SEC2}, we introduce the space and state the main theorem. Section \ref{SEC3} outlines the finite-dimensional systems of ordinary differential equations approximating \eqref{DSDE} and the propagation of the moments of their solutions is also explored. The proof of the existence and mass-conservation property of solutions are discussed in Section \ref{SEC4}. The uniqueness of solutions is examined in Section \ref{SEC5} which is followed by the continuous dependence on initial data discussed in Section \ref{SEC6}. Finally, the large-time behavior of solutions is investigated in Section \ref{SEC7}.

\section{\textbf{ Main Results }} \label{SEC2}
To begin with, we introduce some notations and specify what we mean by a solution to \eqref{DSDE}--\eqref{IC}. The mathematical study of \eqref{DSDE}--\eqref{IC} requires to take into account suitable spaces. Following the usual works related to the coagulation fragmentation area, we will consider the Banach spaces
\begin{align}
\Xi_{\lambda} =\Big\{\xi=(\xi_i)_{i\ge 1} \in \mathbb{R}^{\mathbb{N}}, ~~ \sum_{i=1}^{\infty} i^{\lambda} |\xi_i| < \infty\Big\},\qquad \lambda\ge 0,
\end{align}
with the norm defined by 
\begin{align*}
\| \xi \|_{\lambda} =\sum_{i=1}^{\infty} i^{\lambda}|\xi_i|
\end{align*}

and their positive cones
\begin{align*}
\Xi_{\lambda}^+(T) = \{ \xi =(\xi_i)_{i \geq 1} \in \Xi_{\lambda}: \xi_i \geq 0 \hspace{.2cm}  \text{for each } \hspace{.2cm} i \geq 1\}.
\end{align*}
 Let us now define the notions of solutions to \eqref{DSDE}--\eqref{IC} that we will consider.
\begin{definition} \label{DEF1}
Let $T \in(0,+\infty]$ and $\xi^{in} =(\xi_i^{in})_{i\ge 1}$ be a sequence of non-negative real numbers. A solution $\xi =(\xi_i)_{i \geq 1}$ to \eqref{DSDE}--\eqref{IC} on $[0,T)$ is a sequence of non-negative continuous functions satisfying for each $i\ge 1$ and $t\in(0,T)$
\begin{itemize}
\item[(a)] $\xi_i\in \mathcal{C}([0,T))$, $ \sum_{j=i}^{\infty} \gamma_{i,j}\xi_i \xi_j \in L^1(0,t)$,
\item[(b)] and there holds 
\begin{align*}
\xi_i(t) =\xi_i^{in} + \int_0^t \Big[\xi_{i-1}\sum_{j=1}^{i-1}j \gamma_{i-1,j} \xi_j - \xi_{i}\sum_{j=1}^{i}j \gamma_{i,j} \xi_{i} \xi_j -\sum_{j=i}^{\infty} \gamma_{i,j} \xi_i \xi_j\Big] d\tau.
\end{align*}
\end{itemize}
\end{definition}

In order to show the existence, uniqueness and  mass conservation property of   solutions to \eqref{DSDE})--\eqref{IC}, assume that the collision kernel $\gamma_{i,j}$ is non-negative and  symmetric i.e.\
\begin{align}\label{ASSUM1}
0\leq \gamma_{i,j}= \gamma_{j,i},\hspace{.1cm}i,j \geq 1,
\end{align} and satisfies the following growth condition
\begin{align} \label{ASSUM2}
\text{for all} \hspace{.1cm}i,j \geq 1, \hspace{.1cm} \gamma_{i,j} \leq A (i+j),
\end{align}
where $A$ is a positive constant.

Our existence result then reads as follows.
\begin{theorem}\label{MT}
Assume that the coagulation rate $\gamma_{i,j}$ satisfies assumptions \eqref{ASSUM1}--\eqref{ASSUM2} and $\xi^{in}=(\xi_i^{in})_{i\geq 1} \in \Xi_1^+$. Then there is at least one solution $\xi$ to \eqref{DSDE}--\eqref{IC} satisfying
\begin{align}
\|\xi(t)\|_{\Xi_1} = \|\xi^{in}\|_{\Xi_1}, \qquad t\in [0,+\infty). \label{MC}
\end{align}
\end{theorem}
In other words, the density of the solution $\xi$ is conserved through time evolution.\\

\section{\textbf{Approximating systems}} \label{SEC3}
We demonstrate the existence of solutions to \eqref{DSDE}--\eqref{IC} by taking the limit of solutions to the truncated finite-dimensional system of \eqref{DSDE}--\eqref{IC}. To be more specific, for $ k \geq 2$, let us consider the following truncated system of $k$  ordinary differential equations,

\begin{align}
\frac{d\xi_i^k(t)}{dt}&= \xi_{i-1}^k(t) \sum_{j=1}^{i-1} j \gamma_{i-1,j} \xi_j^k(t)-\xi_i^k(t)  \sum_{j=1}^{i} j \gamma_{i,j} \xi_j^k(t)-\sum_{j=i}^{k} \gamma_{i,j} \xi_i^k(t) \xi_j^k(t), \label{TDSDE}\\
\xi_i^k(0) &= \xi_i^{in} \geq 0, \label{TIC}
\end{align}
for $ i \geq 1$. Next we require following existence result from \cite[Lemma 13]{BAG 2005} for \eqref{TDSDE}--\eqref{TIC}.
\begin{lemma} \label{LEM 1}
For each $k\geq 2$, there exists a unique non-negative solution $\xi^k = (\xi_i^k)_{1\leq i \leq k} $ in $ \mathcal{C}^1([0,T],\mathbb{R}^k)$ to system  \eqref{TDSDE}--\eqref{TIC}. Moreover, we have
\begin{align} 
\sum_{i=1}^k i \xi_i^k(t) \leq \sum_{i=1}^k i \xi_i^{in}, \hspace{.2cm} t\in [0,\infty).\label{TMC}
\end{align}
\end{lemma}
Furthermore, if $(\psi_i)\in \mathbb{R}^k$, there holds 
\begin{align}
\sum_{i=1}^k \psi_i \frac{d\xi_i^k}{dt} = \sum_{i=1}^{k-1} \sum_{j=1}^{i}j \psi_{i+1} \gamma_{i,j} \xi_i^k \xi_j^k -\sum_{i=1}^k \sum_{j=1}^i (j\psi_i +\psi_j)\gamma_{i,j} \xi_i\xi_j. \label{GME}
\end{align} 

We first introduce some notation. We denote by $\mathcal{G}_1$ the set of non-negative and convex functions $G \in C^1([0,+\infty))\cap W_{\text{loc}}^{2,\infty}(0,+\infty)$ such that $G(0)=0$, $G'(0)\geq 0$ and $G'$ is a concave function. We next denote by $\mathcal{G}_{1, \infty}$  the set of functions $G \in \mathcal{G}_1$ satisfying, in addition,
\begin{align}
\lim_{\zeta \to +\infty} G'(\zeta) = \lim_{\zeta\to +\infty} \frac{G(\zeta)}{\zeta} = +\infty.
\end{align}  
\begin{remark} \label{remark1}
It is clear that $\zeta\mapsto \zeta^p$ belongs to $\mathcal{G}_1$ if $p\in [1,2]$ and to $\mathcal{G}_{1,\infty}$ if $p\in (1,2]$. 
\end{remark}	
 \cite[Proposition 7.1.9]{BLL 2019}.
\begin{lemma}
For $G \in \mathcal{G}_1$ and $i,j\ge 1$ there holds 
\begin{align}
(i+j)\big(G(i+j)-G(i) -G(j)\big) \le 2 \big(i G(j) +j G(i)\big). \label{GInequality}
\end{align}
\end{lemma}

We may now state and prove the main result of this section.

\begin{prop} \label{prop1}
Consider $T\in(0,+\infty)$ and $G\in \mathcal{G}_1$. There exists a constant $\kappa(T)$ depending only on $A,G,\|\xi^{in}\|_{\Xi_1}$ and $T$ such that, for each $k\ge 3$, the solution $\xi^k$ to \eqref{TDSDE}--\eqref{TIC} given by Lemma \ref{LEM 1} satisfies
\begin{align}
\sum_{i=1}^k G(i) \xi_i^k(t) \le \kappa(T) \sum_{i=1}^k G(i) \xi_i^{in}, \qquad t\in[0,T]. \label{TPOM}
\end{align}
\end{prop}
\begin{proof}
For $k\ge 3$ and $t\in[0,T]$ we put 
\begin{align*}
M_G^k =\sum_{i=1}^k G(i) \xi_i^k(t).
\end{align*}
We infer from \eqref{ASSUM2} and \eqref{GME} that 
\begin{align*}
\frac{dM_G^k}{dt} \le A \sum_{i=1}^{k}\sum_{j=1}^i (i+j) \big[j((G(i+1)-G(i)) -G(j)\big] \xi_i^k \xi_j^k.
\end{align*}
Now using the convexity of $G$, we obtain
\begin{align*}
j((G(i+1)-G(i)) -G(j)=& jG\Big(\frac{1}{j}(i+j)+\frac{j-1}{j}i\Big)-j G(i)- G(j)\\
&\le j\Big[\frac{1}{j} G(i+j)+\frac{j-1}{j}G(i)\Big]-j G(i)- G(j)\\
&\le G(i+j) - G(i)-G(j).
\end{align*}
The above inequality and \eqref{GInequality} now yield
\begin{align*}
\frac{dM_G^k}{dt} \le& 2A \sum_{i=1}^{k}\sum_{j=1}^i [iG(j) +jG(i)]\xi_i^k \xi_j^k\\
& \le A \|\xi^{in}\|_{\Xi_1} M_G^k,
\end{align*}
which yields \eqref{TPOM} by the Gronwall lemma.
\end{proof}

Next, we recall the following lemma from \cite[Lemma 14]{BAG 2005} which provide the time equicontinuity of $(\xi_i^k)_{k\ge i}$.

\begin{lemma}
Let $i\ge 1$. There exists a constant $\varpi_i$, depending only upon $\|\xi^{in}\|_{\Xi_1}$ and $i$ such that, for each $k \ge i$ 
\begin{align}
\Big|\frac{d\xi_i^k}{dt} \Big| \le \varpi_i, \qquad t\in[0,+\infty). \label{DERVBND}
\end{align}
\end{lemma}

\section{\textbf{Existence of Solutions}} \label{SEC4}

Now we are in a position to prove Theorem \ref{MT}. As a prelude to this, we first recall a refined version of the de la Vallee-Poussin theorem for integrable functions \cite[Theorem 7.1.6]{BLL 2019}.
\begin{theorem} \label{DlVPthm}
Let $(\Sigma,\mathcal{A}, \nu)$ be a measured space and consider a function $\xi \in L^1(\Sigma,\mathcal{A}, \nu)$. Then there exists a function $ G\in \mathcal{G}_{1,\infty}$  such that
\begin{align*}
G(|\xi|) \in L^1(\Sigma,\mathcal{A}, \nu).
\end{align*}
\end{theorem}

\begin{proof}[Proof of Theorem \ref{MT}]
We next apply Theorem~\ref{DlVPthm}, with $\Sigma =\mathbb{N}$ and $\mathcal{A} = 2^{\mathbb{N}}$, the set of all subsets of $\mathbb{N}$. Defining the measure $\nu$ by
\begin{align*}
	\nu(J) =\sum_{i\in J} \xi_i^{\rm{in}},\qquad J \subset \mathbb{N},
\end{align*}
the condition $w^{in}\in Y_1^+$  ensures that $\zeta \mapsto \zeta$  belongs to $L^1(\Sigma,\mathcal{A}, \nu)$. By Theorem~\ref{DlVPthm} there is thus a function $G_0 \in \mathcal{G}_{1,\infty}$ such that $\zeta \mapsto G_0(\zeta)$  belongs to $L^1(\Sigma,\mathcal{A}, \nu)$, that is
\begin{align}
	\mathcal{G}_0 =\sum_{i=1}^{\infty} G_0(i) \xi_i^{\rm{in}} <\infty. \label{Gzero}
\end{align}

In the following we denote by $C$ be a positive constant depending only on $A, \|\xi^{in}\|_{Y_1}$, $G_0$ and $\mathcal{G}_0$. The dependence of $C$ on any additional parameters will be explicitly stated.
By \eqref{TMC} and \eqref{DERVBND}, we infer that the sequence $(\xi_i^k)_{k\ge i}$ is bounded in $W^{1,1}(0,T)$ for each $i\ge 1$ and $T\in(0,+\infty)$. Using the Helly theorem \cite[pp. 372--374]{KF 1970}, we can conclude that there is a subsequence of $(\xi_i^k)_{k\ge i}$ (which we still refer to as $(\xi_i^k){k\geq i}$) and a sequence $\xi=(\xi_i){i\geq 1}$ of functions of locally bounded variation such that
\begin{align}
\lim_{k \to +\infty} \xi_i^k(t) = \xi_i(t) \label{xiLIM}
\end{align}
for each $i\ge 1$ and $t\ge 0$. Clearly $\xi_i(t) \ge 0$ for $i \ge 1$ and $t\ge 0$ and it follows from \eqref{xiLIM} and \eqref{TMC} that $\xi(t) \in \Xi_1^+$ with 
\begin{align}
\|\xi(t)\|_{\Xi_1} \le \|\xi^{in}\|_{\Xi_1}, \qquad  t\ge 0.\label{Y1NORMBOUND}
\end{align}
In addition, since $G_0$ belongs to $\mathcal{G}_{1,\infty}$, we can deduce from \eqref{Gzero} and Proposition~\ref{prop1} that for every $t\geq 0$ and $k\geq 3$, there holds

 	\begin{align}
 	\sum_{i=1}^k G_0(i) \xi_i^k(t) \leq \Lambda(T), \label{G_0bound1k}
 	\end{align}

 	Consider now $T\in (0,+\infty)$ and $m\geq 2$. By \eqref{G_0bound1k},  we have for, $k \geq q+1$ and $t\in[0,T]$
 	\begin{align*}
 	\sum_{i=1}^q G_0(i) \xi_i^l(t) \leq \omega(T).
 	\end{align*}

 	Due to \eqref{xiLIM} we may pass to the limit as $k\to \infty$ in the above estimates, which implies that they both remain valid with $\xi_i^k$ being replaced by $\xi_i$. We next allow $q \to \infty$ and get
\begin{align}
 	\sum_{i=1}^{\infty} G_0(i) \xi_i(t) \leq \omega(T),  \qquad t \in [0,T].\label{G_0bound1}
 	\end{align}
As a result of \eqref{ASSUM2} and \eqref{Y1NORMBOUND} we get that, for each $i \ge  1$,
\begin{align}
\sum_{j=i}^{\infty} \gamma_{i,j} \xi_j \in L^1(0,T). \label{INTEGRABILITY}
\end{align}
We now claim that for each $i \ge 1$, there holds
\begin{align}
\lim_{k\to +\infty} \Big| \sum_{j=i}^{\infty} \gamma_{i,j} \xi_i^k \xi_j^k - \sum_{j=i}^{\infty} \gamma_{i,j} \xi_i \xi_j\Big|_{L^1(0,T)} = 0. \label{CONVERGENCE}
\end{align}

Let us consider now $i\ge 1$ and $q \ge i+1$. Using \eqref{TMC}, \eqref{xiLIM},  and \eqref{Y1NORMBOUND}, along with the Lebesgue dominated convergence theorem, we obtain that
\begin{align}
\lim_{k\to +\infty} \Big| \sum_{j=i}^q a_{i,j} (\xi_i^k \xi_j^k - \xi_i \xi_j )\Big|_{L^1(0,T)} =0. \label{EST11}
\end{align}
Also, we infer from \eqref{ASSUM2}, \eqref{TMC} and \eqref{G_0bound1k} that for each $k \ge m+1$ , 
\begin{align*}
\Big |\sum_{j=q+1}^k \gamma_{i,j} \xi _i^k \xi_j^k\Big|_{L^1(0,T)} \le& Ai\|\xi^{in}\|_1 \Big| \sum_{j=q+1}^kj \xi_j^k\Big|_{L^1(0,T)}\\
& \le \omega(i,T) \sup_{j\ge q} \frac{j}{G_0(j)} \Big| \sum_{j=q+1}^{k} G_0(j) \xi_j^k\Big|_{L^1(0,T)},
\end{align*}
\begin{align}
\Big |\sum_{j=q+1}^k \gamma_{i,j} \xi _i^k \xi_j^k\Big|_{L^1(0,T)} \le \omega(i,T) \sup_{j\ge q}\frac{j}{G_0(j)}. \label{EST12} 
\end{align}
Similarly, \eqref{ASSUM2}, \eqref{Y1NORMBOUND} and \eqref{G_0bound1} entails that
\begin{align}
\Big |\sum_{j=q+1}^{\infty}  \gamma_{i,j} \xi _i \xi_j \Big|_{L^1(0,T)} \le \omega(i,T) \sup_{j\ge q}\frac{j}{G_0(j)}. \label{EST13}
\end{align}
Combining \eqref{EST11}--\eqref{EST13}, we obtain 
\begin{align*}
\limsup_{k\to +\infty} \Big| \sum_{j=i}^k  \gamma_{i,j} \xi _i^k \xi_j^k- \sum_{j=i}^{\infty}  \gamma_{i,j} \xi _i \xi_j\Big|_{L^1(0,T)} \le \omega(i,T) \sup_{j\ge q}\frac{j}{G_0(j)},
\end{align*}
for every $q \ge i+1$. Recalling that $G_0$ belongs to $\mathcal{G}_{1,\infty}$, we can observe that the right-hand side of the above inequality converges to zero as $ q\to +\infty$, as a result, we obtain \eqref{CONVERGENCE}.
With the help of \eqref{TMC}, \eqref{xiLIM}, \eqref{Y1NORMBOUND} and \eqref{CONVERGENCE}, we can easily verify that $\xi_i$ satisfies Definition \ref{DEF1}(b) for each $i \ge 1$. By making use of \eqref{INTEGRABILITY}, the continuity of $\xi_i$ then follows and we have thus shown that $\xi = (\xi_i)$ is a solution to \eqref{DSDE}--\eqref{IC} on $[0,+\infty)$. In order to complete the proof of Theorem \ref{MT}, it remains to prove that \eqref{MC} holds true. Let $t \in (0,+\infty)$. For $ k\ge q \ge 3$, we have \eqref{TMC} that \begin{align*}
|\|\xi(t)\|_1 -\|\xi^{in}\|_1| \le \sum_{i=1}^q i |\xi_i^k(t) -\xi_i^{in}|+ \sum_{i=k+1}^{\infty} i \xi_i^{in} + \sum_{i=q+1}^k i \xi_i^k (t) +  \sum_{i=q+1}^{\infty} i \xi_i (t).
\end{align*}
Subsequently, it can be deduced from \eqref{G_0bound1k} and \eqref{G_0bound1} that
\begin{align*}
|\|\xi(t)\|_1 -\|\xi^{in}\|_1| \le \sum_{i=1}^q i |\xi_i^k(t) -\xi_i^{in}|+ \sum_{i=k+1}^{\infty} i \xi_i^{in} +\omega(T) \sup_{i\ge q} \frac{i}{G_0(i)}.
\end{align*}
Since $\xi\in \Xi_1^+$, we can infer from \eqref{xiLIM} that
\begin{align*}
|\|\xi_i(t)\|_1 -\|\xi_i^{in}\|_1| \le \omega(T) \sup_{i\ge q} \frac{i}{G_0(i)}.
\end{align*}
By noting that $G_0 \in \mathcal{G}_{1,\infty}$, it follows that $\|\xi(t)\|_1 = \|\xi^{in}\|_1$, and the proof of Theorem \ref{MT} is complete.
\end{proof}

Next, we deduce that the solution constructed in Theorem \ref{MT} is, in fact, first-order differentiable.
\begin{corollary}
Assume that the assumptions \eqref{ASSUM1}--\eqref{ASSUM2} are fulfilled. Let $\xi^{in}\in \Xi_1^+$ and consider the solution $\xi=(\xi_i)_{i\ge 1}$ to \eqref{DSDE}--\eqref{IC} on $[0,+\infty)$ given by Theorem \ref{MT}. Then $\xi_i$ is continuously differentiable on $[0,+\infty)$ for each $i \in \mathbb{N}$.
\end{corollary}

\begin{proof}
For $\tau_1, \tau_2 \in [0,T]$ and $q\ge 1$, 
\begin{align*}
\|\xi(\tau_1) -\xi(\tau_2)\|_1 \le& \sum_{i=1}^q i|\xi_i(\tau_1) -\xi_i(\tau_2)| +\sum_{i=q+1}^{\infty} i (\xi_i(\tau_1) + \xi_i(\tau_2) )\\
&\le \sum_{i=1}^q i|\xi_i(\tau_1) -\xi_i(\tau_2)| +\sup_{i\ge q} \frac{i}{G(i)} \sum_{j=q+1}^{\infty} G(i)(\xi_i(\tau_1)+\xi_i(\tau_2))
\end{align*}

As $\xi_i$ is a continuous function for $i\in\{1,2,\cdots,q\}$, we can infer from the above inequality that
\begin{align*}
\limsup_{\tau_1 \to \tau_2}\|\xi(\tau_2) -\xi(\tau_1)\|_1 \le&~~ \sup_{i\ge q} \frac{i}{G_0(i)} \sum_{j=q+1}^{\infty} G_0(i)(\xi_i(\tau_1)+\xi_i(\tau_2))
\end{align*}

By recalling \eqref{G_0bound1} and considering that $G_0$ belongs to $\mathcal{G}_1$, we can take the limit as $q$ approaches infinity, leading to the following result:

Since $G_0 \in \mathcal{G}_1$, recalling \eqref{G_0bound1}, we take the limit as $q\to +\infty$ to obtain
\begin{align*}
\lim_{\tau_1 \to \tau_2}\|\xi(\tau_2) - \xi(\tau_1)\|_1=0
\end{align*}
Next, for $0\le \tau_1 \le \tau_2 <T$ and $i \ge 1$
\begin{align*}
\Bigg| \sum_{j=i}^{\infty} \gamma_{i,j}  \xi_j(\tau_2) - \sum_{j=i}^{\infty} \gamma_{i,j}  \xi_j(\tau_1)\Bigg|\le 2A \|\xi(\tau_2) - \xi(\tau_1)\|_1
\end{align*}
from which the time continuity of $\sum_{j=i}^{\infty} \gamma_{i,j} \xi_j$ follows. It is evident from this that the right-hand side of \eqref{DSDE} is continuous in time, implying the continuity of the derivative of $\xi_i$. Consequently, this ensures the existence of the classical solution.
\end{proof}

We end this section by demonstrating that when the initial data belongs to a certain suitable class and has a finite moment, there exists at least one solution to equations \eqref{DSDE}--\eqref{IC} that maintains the same property for all times.

\begin{prop} \label{prop 3}
Consider $\xi^{\rm{in}} \in \Xi_1^+$ and assume that there is $G\in \mathcal{G}_1$ such that 
\begin{align}
\sum_{i=1}^{\infty} G(i) \xi_i^{\rm{in}}<+\infty. \label{GINMOM}
\end{align}
Then under the assumption \eqref{ASSUM1}--\eqref{ASSUM2} there is at least one solution $\xi$ to \eqref{DSDE}--\eqref{IC} on $[0,+\infty)$  satisfying \eqref{MC} for each $T\in(0,+\infty)$,
\begin{align}
\sup_{t \in [0,T]} \sum_{i=1}^{\infty} G(i) \xi_i(t) <+\infty. \label{GTMOM}
\end{align}
\end{prop}

\begin{proof}
We only need to show that the solution  constructed in the proof of Theorem \ref{MT} enjoys the additional property \eqref{GTMOM}. But, as $G\in \mathcal{G}_1$, \eqref{GTMOM} follows at once from Proposition \ref{prop1} and \eqref{xiLIM}.
\end{proof}

\section{\textbf{Uniqueness of classical solution}}\label{SEC5}

We establish the following identity before proving the uniqueness theorem for the solutions in the space $\Xi_1^+$.
\begin{lemma}\label{Lemma1} 
 Let $ T\in (0,+\infty)$ and $\xi=(\xi_{i})_{i\geq1} \in \Xi_1^+$ be a solution of \eqref{DSDE}--\eqref{IC}. Furthermore, suppose that $\Phi_{i}$ be a sequence having at most polynomial growth. Then, for all $t\in[0,T]$ and $q>1$, we have
\begin{align}\label{Lemma1equation}
\frac{d}{dt}\sum_{i=1}^{q}\Phi_{i}\xi_{i}(t)=&\sum_{P_{1}}j\Phi_{i+1} \gamma_{i,j} \xi_{i}(t)\xi_{j}(t)
-\sum_{P_{2}}(j\Phi_{i}+\Phi_{j}) \gamma_{i,j}  \xi_{i}(t)\xi_{j}(t)\nonumber\\
&-\sum_{P_{3}}\Phi_{j} \gamma_{i,j} \xi_{i}(t)\xi_{j}(t),
\end{align}
where
\begin{align*}
P_{1}=&\{(i,j): 1\leq i\leq q-1,~~ 1\leq j\leq i\},\\
P_{2}=&\{(i,j): 1\leq i\leq q, ~~1\leq j\leq i\},\\
P_{3}=&\{(i,j): i\geq q+1,~~ 1\leq j\leq q\}.
\end{align*}
\end{lemma}
\begin{proof} From \eqref{DSDE}, we have
\begin{align*}
 \frac{d \xi_{i}(t) }{d t} = \xi_{i-1}(t)\sum_{j=1}^{i-1}j \gamma_{i-1,j} \xi_{j}(t) - \xi_{i}(t)\sum_{j=1}^{i}j \gamma_{i,j}  \xi_{j}(t) - \xi_{i}(t)\sum_{j=i}^{\infty}\gamma_{i,j} \xi_{j}(t).
\end{align*}
Multiplying by $\Phi_{i}$ in the above equation and then taking summation from $i=1$ to $i=q$ on both sides, we obtain
\begin{align}\label{Lemma1eq2}
 \frac{d}{d t}\sum_{i=1}^{q}\Phi_{i}\xi_{i}(t)= & \sum_{i=1}^{q}\sum_{j=1}^{i-1}j\Phi_{i} \gamma_{i-1,j}\xi_{i-1}(t) \xi_{j}(t) - \sum_{i=1}^{q}\sum_{j=1}^{i}j\Phi_{i}\gamma_{i,j}\xi_{i}(t)\xi_{j}(t)\nonumber\\ & - \sum_{i=1}^{q}\sum_{j=i}^{\infty}\Phi_{i}\gamma_{i,j}\xi_{i}(t)\xi_{j}(t).
\end{align}
 On changing the order of summation in the first  and the last term as the r.h.s. to \eqref{Lemma1eq2}, we get

\begin{align}\label{Lemma1eq3}
 \frac{d}{d t}\sum_{i=1}^{q}\Phi_{i} \xi_{i}(t)= &\sum_{j=1}^{q-1}\sum_{i=j+1}^{q}j \Phi_{i}\gamma_{i-1,j}\xi_{i-1}(t) \xi_{j}(t)
 -\sum_{i=1}^{q}\sum_{j=1}^{i}j\Phi_{i}\gamma_{i,j}\xi_{i}(t)\xi_{j}(t)\nonumber\\
 &-\sum_{j=1}^{q}\sum_{i=1}^{j}\Phi_{i}\gamma_{i,j}\xi_{i}(t)\xi_{j}(t)
-\sum_{j=q+1}^{\infty}\sum_{i=1}^{q}\Phi_{i}\gamma_{i,j}\xi_{i}(t)\xi_{j}(t).
 \end{align}

 Next, in the first summation, we replace $i-1$ with $i$, and in the last two summations on the r.h.s., we exchange $i$ and $j$ and use the symmetry of $\gamma_{i,j}$ to \eqref{Lemma1eq3}, to obtain
 
 \begin{align}\label{Lemma1eq4}
  \frac{d}{d t} \sum_{i=1}^{q}\Phi_{i} \xi_{i}(t) =&\sum_{j=1}^{q-1}\sum_{i=j}^{q-1}j \Phi_{i+1}\gamma_{i,j}\xi_{i}(t)\xi_{j}(t)
  -\sum_{i=1}^{q}\sum_{j=1}^{i}j\Phi_{i}\gamma_{i,j}\xi_{i}(t) \xi_{j}(t)\nonumber\\
  &-\sum_{i=1}^{q}\sum_{j=1}^{i}\Phi_{j}\gamma_{i,j}\xi_{i}(t)\xi_{j}(t)
  -\sum_{i=q+1}^{\infty}\sum_{j=1}^{q}\Phi_{j}\gamma_{i,j}\xi_{i}(t)\xi_{j}(t).
\end{align}

Again, by using a change in the order of summations to the first summation on the r.h.s. to \eqref{Lemma1eq4}, we have
\begin{align}\label{Lemma1eq5}
   \frac{d}{d t}\sum_{i=1}^{q}\Phi_{i} \xi_{i}(t)=&\sum_{i=1}^{q-1}\sum_{j=1}^{i}j \Phi_{i+1}\gamma_{i,j}\xi_{i}(t)\xi_{j}(t)
   -\sum_{i=1}^{q}\sum_{j=1}^{i}(j\Phi_{i}+\Phi_{j})\gamma_{i,j}\xi_{i}(t) \xi_{j}(t)\nonumber\\
   &-\sum_{i=q+1}^{\infty}\sum_{j=1}^{q}\Phi_{j}\gamma_{i,j}\xi_{i}(t) \xi_{j}(t).
\end{align}

which clearly implies that \eqref{Lemma1equation} holds.
\end{proof}

Now we are in a position to give the proof of the uniqueness of the solutions.

\begin{theorem}
Consider $\xi^{in}\in \Xi_1^+$ and assume that the coagulation kernel satisfies the assumptions \eqref{ASSUM1}. Assume further that there are $\delta \in [0,1]$ and $A_{\delta}>0$ such that
\begin{align}
\sum_{i=1}^{\infty} i^{1+\delta} \xi_i^{in} <+\infty ~~~~~~ \text{and}~~~~~~ a_{i,j} \le A(i^{\delta} +j^{\delta}),~~~i,j\ge 1.
\end{align}
Then there is one and only one solution $\xi$ to \eqref{DSDE}--\eqref{IC} on $[0,+\infty)$ satisfying \eqref{MC} and, for each $T\in(0,+\infty)$,
\begin{align}
\sup_{t\in[0,T]}\sum_{i=1}^{\infty} i^{1+\delta} \xi_i(t) <+\infty. \label{HMF}
\end{align} 
\end{theorem}
\begin{proof}
As $\varsigma\mapsto \varsigma^{1+\delta} $ belongs to $\mathcal{G}_1$, the existence of solution to \eqref{DSDE}--\eqref{IC} on $[0,+\infty)$ satisfying \eqref{MC} and \eqref{HMF} follows from Proposition \ref{prop 3}.

Let  $\xi$ and $\eta$ be two distinct solutions of \eqref{DSDE}--\eqref{IC} having initial condition $\xi_i(0) =  \eta_i(0)$ for all $i \geq 1$. Let $\pi(t) = \xi(t) -\eta(t)$, for $ t \in [0,T]$. Then we consider the following function as

\begin{align*}
\mathcal{G}(t)= \sum_{i=1}^{q}i \big|\xi_i(t) - \eta_i(t) \big|  =\sum_{i=1}^{q} i|\pi_i(t)|
\end{align*}
\begin{align}
\frac{d\pi_i}{dt}=& \sum_{j=1}^{i-1}j \gamma_{i-1,j}  \big[ \xi_{i-1}(t) \xi_j(t) -\eta_{i-1}(t) \eta_j(t) \big] -\sum_{j=1}^i j  \gamma_{i,j}  \big[ \xi_{i}(t) \xi_j(t) -\eta_{i}(t) \eta_j(t) \big] \nonumber\\
& + \sum_{j=i}^{\infty} \gamma_{i,j}\big[ \xi_{i}(t) \xi_j(t) -\eta_{i}(t) \eta_j(t) \big].\label{EST5}
\end{align}

Now using equation \eqref{Lemma1equation} for $\xi_i$ and $\eta_i$ and taking the difference with $ \Phi_i = i\sgn(\pi_i(t))$, we have

\begin{align}
\sum_{i=1}^{q}i \big| \xi_i(t) - \eta_i(t) \big| = &\int_0^t  \Bigg[\sum_{i=1}^{q-1}\sum_{j=1}^{i}(i+1)\sgn(\pi_{i+1}(s)) j \gamma_{i,j} \big[\xi_{i}(s) \xi_j(s) -\eta_{i}(s) \eta_j(s) \big] \nonumber\\
&-\sum_{i=1}^q \sum_{j=1}^i (ji\sgn(\pi_i(s))+ j \sgn(\pi_j(s))  \gamma_{i,j}  \big[\xi_{i}(s) \xi_j(s) -\eta_{i}(s) \eta_j(s) \big] \nonumber\\
&-\sum_{i=q+1}^{\infty} \sum_{j=1}^{q} j\sgn(\pi_j(s))  \gamma_{i,j}\big[ \xi_{i}(s) \xi_j(s) -\eta_{i}(s) \eta_j(s) \big] \Big ] ds. \label{EST0}
\end{align}
Since 
\begin{align*}
\xi_i(t) \xi_j(t) - \eta_i(t) \eta_j(t) = \xi_i(t) \pi_j(t)+ \eta_j(t) \pi_i(t),
\end{align*}

the identity \eqref{EST0} can be written as 
\begin{align}
\sum_{i=1}^{q}i\big|\xi_i(t) - \eta_i(t) \big| = &\int_0^t \Bigg[\sum_{i=1}^{q-1}\sum_{j=1}^{i}(i+1) j\sgn(\pi_{i+1}(s))  \gamma_{i,j} \big[\xi_{i}(s) \pi_j(s)+\eta_{j}(s) \pi_i(s) \big] \nonumber\\
&-\sum_{i=1}^{q}\sum_{j=1}^i (ji\sgn(\pi_i(s))+ j\sgn(\pi_j(s))\gamma_{i,j}  \big[ \xi_{i}(s) \pi_j(s) +\eta_{j}(s) \pi_i(s) \big] \nonumber\\
&- \sum_{i=q+1}^{\infty}\sum_{j=1}^{q}j \sgn(\pi_j(s)) \gamma_{i,j}\big[ \xi_{i}(s) \pi_j(s)+\eta_{j}(s) \pi_i(s) \big] \Bigg ] ds:= I_1 + I_2 +I_3. \label{I123}
\end{align}

Now, we estimate the terms $I_1$, $I_2$ and $I_3$ separately. Let us first consider the term $I_1$ as 
\begin{align}
I_1 =& \int_0^t \sum_{i=1}^{q-1}\sum_{j=1}^{i}(i+1) j\sgn(\pi_{i+1}(s))  \gamma_{i,j} \big[\xi_{i}(s) \pi_j(s)+\eta_{j}(s) \pi_i(s) \big] ds \nonumber\\
& \leq \int_0^t \sum_{i=1}^{q-1}\sum_{j=1}^{i} ij \gamma_{i,j} \xi_{i}(s)|\pi_j(s)| ds +\int_0^t \sum_{i=1}^{q-1}\sum_{j=1}^{i} ij \gamma_{i,j} \eta_{j}(s)|\pi_i(s)| ds \nonumber\\
&+ \int_0^t \sum_{i=1}^{q-1}\sum_{j=1}^{i} j \gamma_{i,j} \xi_{i}(s)|\pi_j(s)| ds +\int_0^t \sum_{i=1}^{q-1}\sum_{j=1}^{i} j \gamma_{i,j} \eta_{j}(s)|\pi_i(s)| ds. \label{I_1}
\end{align}
Similarly $I_2$ can be evaluated as
\begin{align}
I_2 =&-\sum_{i=1}^{q}\sum_{j=1}^i (j i \sgn(\pi_i(s))+ j\sgn(\pi_j(s))\gamma_{i,j}  \big[ \xi_{i}(s) \pi_j(s) +\eta_{j}(s) \pi_i(s) \big]ds \nonumber\\
& \leq \int_0^t \sum_{i=1}^{q}\sum_{j=1}^i ij \gamma_{i,j}  \xi_{i}(s)|\pi_j(s)| ds-  \int_0^t \sum_{i=1}^{q}\sum_{j=1}^i ij \gamma_{i,j}\eta_{j}(s) |\pi_i(s)|ds \nonumber \\
&- \int_0^t \sum_{i=1}^{q}\sum_{j=1}^i j\gamma_{i,j}   \xi_{i}(s) |\pi_j(s)|ds+ \int_0^t  \sum_{i=1}^{q}\sum_{j=1}^ij\gamma_{i,j}\eta_{j}(s) |\pi_i(s)|ds. \label{I_2}
\end{align}

Adding \eqref{I_1} and \eqref{I_2}, we get

\begin{align}
I_1 + I_2 &\leq 2 \int_0^t \sum_{i=1}^{q}\sum_{j=1}^{i} i j \gamma_{i,j} \xi_{i}(s)|\pi_j(s)| ds + 2  \int_0^t  \sum_{i=1}^{q}\sum_{j=1}^ij\gamma_{i,j}\eta_{j}(s) |\pi_i(s)|ds \nonumber \\
&\leq  4A (\Delta_{\lambda}(T)+\Delta_1) \int_0^t \sum_{i=1}^q i |\pi_i(s)| ds \label{I_1+I_2}
\end{align}
Finally, $I_3$ estimated as 
\begin{align}
I_3 =&-\int_0^t \sum_{i=q+1}^{\infty}\sum_{j=1}^{q}j \sgn(\pi_j(s)) \gamma_{i,j}\big[ \xi_{i}(s) \pi_j(s)+\eta_{j}(s) \pi_i(s) \big] \Bigg ] ds \nonumber \\
& \leq -A \int_0^t  \sum_{i=q+1}^{\infty}\sum_{j=1}^{q}  j(i^{\delta}+j^{\delta})  \xi_{i}(s) |\pi_j(s)| ds +A \int_0^t  \sum_{i=q+1}^{\infty}\sum_{j=1}^{q}  j(i^{\delta}+j^{\delta}) \eta_{j}(s) |\pi_i(s)| ds \nonumber \\
& \leq A (\Delta_{\lambda}(T)+\Delta_1) \int_0^t  \sum_{i=q+1}^{\infty} i  (\xi_{i}(s) + \eta_{i}(s)) ds\label{I_3}
\end{align}

Using \eqref{I_1}, \eqref{I_2} and \eqref{I_3} in \eqref{I123}, we obtain
\begin{align}
\sum_{i=1}^{q}i\big|\pi_i(t) \big| \leq & 4A (\Delta_{\lambda}(T)+\Delta_1) \int_0^t \sum_{i=1}^q i |\pi_i(s)| ds + A (\Delta_{\lambda}(T)+\Delta_1) \int_0^t  \sum_{i=q+1}^{\infty} i  (\xi_{i}(s) + \eta_{i}(s)) ds \label{EST6}
\end{align}

Therefore, using \eqref{HMF}, we may pass to the limit as $q \to \infty$ in \eqref{EST6}, we obtain
\begin{align*}
\sum_{i=1}^{\infty}i|\pi_i(t)| \leq & 4A (\Delta_{\lambda}(T)+\Delta_1) \int_0^t \sum_{i=1}^{\infty} i |\pi_i(s)| ds
\end{align*} 
Since $\pi_i(0) =0$, then, by the application of Gronwall's lemma, we conclude that
\begin{align*}
\sum_{i=1}^{\infty} i |\pi_i(t)|= 0, \hspace{.4cm} \forall t \in [0,T],
\end{align*}
hence $\pi_i=0$, for all $i\geq 1$ and $ t \in[0,T]$, which proves the uniqueness.
\end{proof}

\section{\textbf{Continuous dependence on initial data}} \label{SEC6}
Concerning the continuous dependence relative to the initial condition,
we prove the following  result.
\begin{prop}
If assumptions \eqref{ASSUM1}--\eqref{ASSUM2} holds and if $\xi$ and $\eta$ are solutions to \eqref{DSDE} in $\Xi_1^+$ satisfying $\xi(0)= \xi^{in}$ and $ \eta(0) = \eta^{in}$ then, for each $t\in [0,T]$, there is a positive constant $\Upsilon(T,\Delta_1, \Delta_{\lambda}(T))$ such that
\begin{align}
\|\xi-\eta\|_1 \leq  \Upsilon(T,\Delta_1, \Delta_{\lambda}(T)) \|\xi^{in} -\eta^{in}\|_1. \label{CD}
\end{align}  
\end{prop}

\begin{proof}
Defining $\pi(t) =\xi(t) -\eta(t)$ and using the same estimates as in the proof of Theorem \ref{UNIQ} to obtain
\begin{align*}
\sum_{i=1}^{q}i\big|\pi_i(t) \big| \leq & \sum_{i=1}^{q}i\big|\pi_i(0)\big| + 4A (\Delta_{\lambda}(T)+\Delta_1) \int_0^t \sum_{i=1}^q i |\pi_i(s)| ds \\
&+ A (\Delta_{\lambda}(T)+\Delta_1) \int_0^t  \sum_{i=q+1}^{\infty} i  (\xi_{i}(s) + \eta_{i}(s)) ds.
\end{align*}
As a consequence, by making $q\to \infty$ and using the similar arguments as in the proof of uniqueness, we get
\begin{align*}
\sum_{i=1}^{\infty}i\big|\pi_i(t) \big| \leq & \sum_{i=1}^{\infty}i\big|\pi_i(0)\big| + 4A (\Delta_{\lambda}(T)+\Delta_1) \int_0^t \sum_{i=1}^{\infty} i |\pi_i(s)| ds.
\end{align*}
By using the Gronwall's lemma and then taking supremum over $t$, we obtain \eqref{CD}.
\end{proof}

\section{\textbf{Asymptotic behavior of solutions}}\label{SEC7}

In this section, we investigate the behaviour of the solutions to \eqref{DSDE}--\eqref{IC} as $t\rightarrow\infty$ and here we follow the proof from \cite[Theorem 4.3]{COSTA 94}. 
\begin{thm}\label{thm1}
For $T \in (0,+\infty)$ and $\xi^{in} = (\xi_i^{in})_{i \geq 1} \in \Xi_1^+$. let $\xi=(\xi_{i})_{i\geq1} \in \Xi_1^+$ be a solution of \eqref{DSDE}--\eqref{IC}, then there is $\xi^{\infty}=(\xi_{i}^{\infty})_{i\geq 1}\in \Xi_1^+$ such that
\begin{align}\label{thmeq1}
\lim_{t\rightarrow+\infty}\xi_{i}(t)=\xi_{i}^{\infty},~~i\geq1.
\end{align}
Moreover, if $\gamma_{i,i}>0$ for all $i\geq1$, then we have
\begin{align}\label{thmeq2}
\xi_{i}^{\infty}=0, \hspace{.2cm} \text{for all} \hspace{.2cm}i \geq 1.
\end{align}
\end{thm}
\begin{proof}
Consider $q\geq1,~\tau\geq0$ and $s\geq \tau$. Using $\Phi\equiv1$ into Lemma \ref{Lemma1} we obtain that
\begin{align}
   \frac{d}{d t}\sum_{i=1}^{q}\xi_{i}(t)=&\sum_{i=1}^{q-1}\sum_{j=1}^{i}j\gamma_{i,j}\xi_{i}(t)\xi_{j}(t)
   -\sum_{i=1}^{q}\sum_{j=1}^{i}(j+1)\gamma_{i,j}\xi_{j}(t) \xi_{k}(t)\nonumber\\
   &-\sum_{i=q+1}^{\infty}\sum_{j=1}^{q}\gamma_{i,j}\xi_{i}(t) \xi_{j}(t)\nonumber\\
   \leq&-\sum_{i=1}^{q}\sum_{j=1}^{i}\gamma_{i,j}\xi_{i}(t) \xi_{j}(t)-\sum_{i=q+1}^{\infty}\sum_{j=1}^{q}\gamma_{i,j}\xi_{i}(t) \xi_{j}(t)\nonumber\\
   \leq&  -\frac{1}{2} \sum_{i=1}^{q}\sum_{j=1}^{q}\gamma_{i,j}\xi_{i}(t) \xi_{j}(t) \label{EST7}\\
   \leq &0. \nonumber
\end{align}

Taking integration  with respect to $t$ from $s$ to $\tau$, we have 
\begin{align}\label{thm1eq1}
\sum_{i=1}^{q}\xi_{j}(s)\geq\sum_{i=1}^{q}\xi_{j}(\tau).
\end{align}
Hence, for each $q\geq1$, the function $f_{q}:t\mapsto \sum_{i=1}^{q}\xi_{i}(t)$ is a non-increasing and non-negative function of time.

 Hence, there exists a positive constant $\overline{f_{q}}$ such that
$$f_{q}(t)\rightarrow \overline{f_{q}}~~~\mbox{as}~~t\rightarrow+\infty,$$
and thus
$$\xi_{q}(t)\rightarrow \xi_{q}^{\infty} ~~~\mbox{as}~~t\rightarrow+\infty,$$
with $\xi_{1}^{\infty}=\overline{f}_{1}$ and $\xi_{q}^{\infty}=\overline{f}_{q}-\overline{f}_{q-1}\geq0$ for $q\geq2$. Moreover, we conclude from $\xi= (\xi_i)_{i\geq 1}\in \Gamma_{\lambda}^+(T)$ that
$$\sum_{i=1}^{q}i^{\lambda}\xi_{i}(t)\leq\sup_{t\in[0,+\infty)}\|\xi(t)\|_{\lambda}<+\infty,$$
for each $q\geq1$ and $t\in[0, +\infty)$. Hence
$$\sum_{i=1}^{\infty}i^{\lambda}\xi_{i}^{\infty}<+\infty.$$


Now, we prove \eqref{thmeq2}, i.e. ${\xi}_{q}^{\infty}=0$ for all $q\geq1$. We first set $q=1$ into Lemma \ref{Lemma1} and obtain
\begin{align*}
\xi_{1}(t+\tau)-\xi_{1}(t)=&- \int_{t}^{t+\tau}[\xi_{1}^{2}(s)\gamma_{1,1} - \xi_{1}(s)\sum_{k=1}^{\infty}\gamma_{j,k}\xi_{k}(s)]ds\\
=& \int_{t}^{t+\tau}[-2\xi_{1}^{2}(s)\gamma_{1,1} - \xi_{1}(s)\sum_{k=2}^{\infty}\gamma_{1,k}\xi_{k}(s)]ds\\
\leq& -\int_{t}^{t+\tau}\Psi_{1,1}\xi_{1}^{2}(s)ds,\\
\end{align*}
for all $t\geq0$ and $\tau>0$. Now, letting $t\rightarrow\infty$, we get
$$\lim_{t\rightarrow\infty}\int_{t}^{t+\tau}\gamma_{1,1}\xi_{1}^{2}(s)ds=0$$
and this implies $\xi_{1}^{\infty}=0$, because otherwise there would exist a positive constant $0<\vartheta_{1}<\xi_{1}^{\infty}$ such that, for all sufficiently large $t,~\xi_{1}(t)>\vartheta_{1}$ and
 $$\lim_{t\rightarrow\infty}\int_{t}^{t+\tau}\xi_{1}^{2}(s)ds\geq \gamma_{1,1}\vartheta_{1}^{2}\tau>0,$$
 a contradiction. Next, we set $q=2$ into Lemma \ref{Lemma1} and get
\begin{align*}
\xi_{2}(t+\tau)-\xi_{2}(t)=&\int_{t}^{t+\tau}[\xi^{2}_{1}(s)\gamma_{1,1} - \xi_{2}(s)\sum_{k=1}^{2}k \gamma_{2,k} \xi_{k}(s) - \xi_{2}(s)\sum_{k=2}^{\infty}\gamma_{2,k}\xi_{k}(s)]ds\\
=&\int_{t}^{t+\tau}[\xi^{2}_{1}(s)\gamma_{1,1} - 3\xi_{2}^{2}(s)\gamma_{2,2}- \gamma_{1,2}\xi_{1}(s)\xi_{2}(s) - \xi_{2}(s)\sum_{k=3}^{\infty}\gamma_{2,k}\xi_{k}(s)]ds\\
\leq& \int_{t}^{t+\tau}(\xi^{2}_{1}(s)\gamma_{1,1}-\xi_{2}^{2}(s)\gamma_{2,2})ds.
\end{align*}
Now, letting $t\rightarrow\infty$ and obtain
\begin{align*}
 0\leq&\lim_{t\rightarrow\infty}\int_{t}^{t+\tau}(\xi_{1}^{2}(s)\Psi_{1,1}-\xi_{2}^{2}(s)\gamma_{2,2})ds
\end{align*}
and since $\xi_{1}(t)\rightarrow0$ as $t\rightarrow\infty$ and $(\gamma_{i,i})_{i\geq1}>0$, we have
$$\lim_{t\rightarrow\infty}\int_{t}^{t+\tau}\gamma_{2,2}\xi_{2}^{2}(s)ds=0,$$
and this implies $\xi_{2}^{\infty}= 0$, because otherwise there would exist a positive constant $0<\vartheta_{2}<\xi_{2}^{\infty}$ such that, for all sufficiently large $t,~\xi_{2}(t)>\vartheta_{2}$ and
 $$\lim_{t\rightarrow\infty}\int_{t}^{t+\tau}\xi_{2}^{2}(s)ds\geq \gamma_{2,2}\vartheta_{2}^{2}\tau>0,$$
 a contradiction.\\
 Proceeding by induction, assuming $\xi_{1}^{\infty}=\cdots=\xi_{q-1}^{\infty}=0$ we prove $\xi_{q}^{\infty}=0$:
 \begin{align*}
 0=&\lim_{t\rightarrow\infty}\xi_{q}(t+\tau)-\xi_{q}(t)\\
 =&\lim_{t\rightarrow\infty}\bigg[\int_{t}^{t+\tau}\bigg(\xi_{q-1}(s)\sum_{k=1}^{q-1}k \gamma_{q-1,k} \xi_{k}(s) - \xi_{q}(s)\sum_{k=1}^{q}k \gamma_{q,k}  \xi_{k}(s) - \xi_{q}(s)\sum_{k=q}^{\infty}\gamma_{q,k}\xi_{k}(s)\bigg)ds\bigg]\\
 \leq&-\Psi_{q,q}\lim_{t\rightarrow\infty}\int_{t}^{t+\tau}(\xi_{q}(t))^{2}ds,
 \end{align*}
 and the conclusion follows as before.
  Hence the proof of Theorem \ref{thm1} is completed.
\end{proof}

Finally, in the next proposition, we will show that the total number of particles goes to zero as time increases to infinity.
\begin{prop}
For $T \in (0,+\infty)$ and $\xi^{in} = (\xi_i^{in})_{i \geq 1} \in \Xi_1^+$. Let $\xi=(\xi_{i})_{i\geq1} \in  \Xi_1^+ $ be a solution of \eqref{DSDE}--\eqref{IC}. Further assume that for some $\zeta>0$
\begin{align} 
\gamma_{i,j} \geq \zeta, \hspace{.2cm}  \text{and}\hspace{.1cm} i,j \geq 1. \label{GAMMALB}
\end{align}
Then \begin{align*}
\lim_{t \to \infty} \sum_{i=1}^{\infty} \xi_i(t) = 0. 
\end{align*}
\end{prop}
 
 \begin{proof}
From \eqref{EST7}, it follows that
\begin{align*}
\sum_{i=1}^{q} \xi_i(\tau) + \frac{1}{2}\int_s^{\tau} \sum_{i=1}^{q}\sum_{j=1}^{q}\gamma_{i,j}\xi_{i}(t) \xi_{j}(t)dt \leq \sum_{i=1}^{q} \xi_i(s).  
\end{align*}
Now, the growth conditions \eqref{ASSUM2} and \eqref{GAMMALB}  allow to pass to the limit as $q\to \infty$ in the above equality, we thus obtain 
\begin{align*}
\sum_{i=1}^{\infty} \xi_i(\tau) + \frac{1}{2}\int_s^{\tau} \sum_{i=1}^{\infty}\sum_{j=1}^{\infty}\gamma_{i,j}\xi_{i}(t) \xi_{j}(t)dt \leq \sum_{i=1}^{\infty} \xi_i(s), 
\end{align*}
which implies that
\begin{align*}
\int_s^{\tau} \Big(\sum_{i=1}^{\infty}\xi_{i}(t) \Big)^2 dt \leq \frac{2}{\zeta}\sum_{i=1}^{\infty} \xi_i(s).
\end{align*}

 Next, we deduce from the previous estimate (with $s=0 $ and $\tau=+\infty$) that
\begin{align*}
\int_0^{\infty} \Big(\sum_{i=1}^{\infty} \xi_i(t)\Big)^2 dt \leq \frac{2}{\zeta}\sum_{i=1}^{\infty} \xi_i^{in} \leq \frac{2}{\zeta} \sum_{i=1}^{\infty} i \xi_i^{in} < + \infty.
\end{align*}

Recalling \eqref{thm1eq1}, we realize that total number of particles $M_0$ is a non-increasing and non-negative function of time which also belongs to $L^2(0, + \infty)$. Therefore, we obtain
\begin{align*}
\lim_{t\to \infty}\sum_{i=1}^{\infty} \xi_i(t) =0.
\end{align*}

\end{proof}

\vspace{1cm}

\subsection*{Acknowledgments}
 This work is  partially supported by Department of Science \& Technology (DST), India-Deutscher Akademischer Austauschdienst (DAAD) within the Indo-German joint project entitled ''Analysis and Numerical Methods for Population Balance Equations''. The authors would like to thank Prof. Philippe Lauren\c{c}ot for helpful discussion.
\subsection*{\textbf{Funding Information}}
MA would like to thank the University Grant Commission (UGC), India for granting the Ph.D. fellowship through Grant No. 416611.
\subsection*{\textbf{Competing Interests}}
The authors declare that they have no conflict of interests.
\subsection*{Author Contribution}
All the authors have contributed equally.

\end{document}